\documentclass[12pt]{amsart}
\usepackage {amsfonts,amscd,amssymb,amsmath,tabularx}
\usepackage{enumerate}

\makeatletter
\def\@captype{figure}
\makeatother

\newtheorem {prop}{Proposition}[section]
\newtheorem {lemme}[prop]{Lemma}
\newtheorem {theoreme}[prop]{Theorem}
\newtheorem {corollaire}[prop]{Corollary}

\newtheorem {exemple}[prop]{Example}

\newtheorem {remarque}[prop]{Remark}

\newtheorem {prop*}{Proposition}
\newtheorem {thm*}[prop*]{Theorem}

\newcommand\cset{\mathbb C}
\newcommand\zset{\mathbb Z}
\newcommand\nset{\mathbb N}

\newcommand\ilie {\mathfrak {i}}
\newcommand\blie {\mathfrak {b}}
\newcommand\glie {\mathfrak {g}}
\newcommand\hlie {\mathfrak {h}}
\newcommand\plie {\mathfrak {p}}

\newcommand \calf {{\mathcal F}}

\newcommand \calc {{\mathcal C}}
\newcommand \cala {{\mathcal A}}

\newcommand \calu {{\mathcal U}}
\newcommand \call {{\mathcal L}}

\def\htrait{\omit\hrulefill}

\def\vtrait{\vrule height 10pt depth 6pt} 
\def\sboite#1{\vtrait\hbox to 16pt{\hfil #1\hfil}} 
\def\hboite#1{\hbox to 16pt{\hfil #1\hfil}} 

\def\vt{\vrule height 10pt depth 6pt} 
\def\sb#1{\vt\hbox to 36pt{\hfil #1\hfil}} 
\def\hb#1{\hbox to 36pt{\hfil #1\hfil}}

\newenvironment{Tiles}[1]{\setlength{\unitlength}{#1pt}\begin{array}{l}}{\end{array}}

\newcommand{\Dbloc}[1]{\begin{picture}(20,20)#1\end{picture}}

\newcommand{\Dtext}[2]{\makebox(20,20)[#1]{\scriptsize $#2$}}
\newcommand{\Ddoubletext}[2]{\makebox(40,20)[#1]{\scriptsize $#2$}}

\newcommand{\Ttop}{\put(0,20){\line(1,0){20}}}
\newcommand{\Tbottom}{\put(0,0){\line(1,0){20}}}
\newcommand{\Tleft}{\put(0,0){\line(0,1){20}}}
\newcommand{\Tright}{\put(20,0){\line(0,1){20}}}
\newcommand{\Tantidiagonal}{\put(0,0){\line(1,1){20}}}

\newcommand{\Ttopdots}{\put(4,20){\circle*{0.1}}\put(8,20){\circle*{0.1}}\put(12,20){\circle*{0.1}}\put(16,20){\circle*{0.1}}}
\newcommand{\Tbottomdots}{\put(4,0){\circle*{0.1}}\put(8,0){\circle*{0.1}}\put(12,0){\circle*{0.1}}\put(16,0){\circle*{0.1}}}
\newcommand{\Tleftdots}{\put(0,4){\circle*{0.1}}\put(0,8){\circle*{0.1}}\put(0,12){\circle*{0.1}}\put(0,16){\circle*{0.1}}}
\newcommand{\Trightdots}{\put(20,4){\circle*{0.1}}\put(20,8){\circle*{0.1}}\put(20,12){\circle*{0.1}}\put(20,16){\circle*{0.1}}}

\newcommand{\Ttopleftdot}{\put(0,20){\circle*{0.1}}}
\newcommand{\Ttoprightdot}{\put(20,20){\circle*{0.1}}}
\newcommand{\Tbottomleftdot}{\put(0,0){\circle*{0.1}}}
\newcommand{\Tbottomrightdot}{\put(20,0){\circle*{0.1}}}

\newcommand{\Dskip}{\\ [-4.5pt]}
\newcommand{\Dspace}{\Dbloc{}}

\catcode`\@=11

\thinlines
\newskip\Einheit \Einheit=0.6cm
\newcount\xcoord \newcount\ycoord
\newdimen\xdim \newdimen\ydim \newdimen\PfadD@cke \newdimen\Pfadd@cke
\def\PfadDicke#1{\PfadD@cke#1 \divide\PfadD@cke by2 \Pfadd@cke\PfadD@cke \multiply\PfadD@cke by2}
\long\def\LOOP#1\REPEAT{\def\BODY{#1}\ITERATE}
\def\ITERATE{\BODY \let\next\ITERATE \else\let\next\relax\fi \next}
\let\REPEAT=\fi
\def\Punkt{\hbox{\raise-2pt\hbox to0pt{\hss\scriptsize$\bullet$\hss}}}
\def\DuennPunkt(#1,#2){\unskip
  \raise#2 \Einheit\hbox to0pt{\hskip#1 \Einheit
          \raise-2.5pt\hbox to0pt{\hss\normalsize$\bullet$\hss}\hss}}
\def\NormalPunkt(#1,#2){\unskip
  \raise#2 \Einheit\hbox to0pt{\hskip#1 \Einheit
          \raise-3pt\hbox to0pt{\hss\large$\bullet$\hss}\hss}}
\def\DickPunkt(#1,#2){\unskip
  \raise#2 \Einheit\hbox to0pt{\hskip#1 \Einheit
          \raise-4pt\hbox to0pt{\hss\Large$\bullet$\hss}\hss}}
\def\Kreis(#1,#2){\unskip
  \raise#2 \Einheit\hbox to0pt{\hskip#1 \Einheit
          \raise-4pt\hbox to0pt{\hss\Large$\circ$\hss}\hss}}
\def\Diagonale(#1,#2)#3{\unskip\leavevmode
  \xcoord#1\relax \ycoord#2\relax
      \raise\ycoord \Einheit\hbox to0pt{\hskip\xcoord \Einheit
         \unitlength\Einheit
         \line(1,1){#3}\hss}}
\def\AntiDiagonale(#1,#2)#3{\unskip\leavevmode
  \xcoord#1\relax \ycoord#2\relax \advance\xcoord by -0.05\relax
      \raise\ycoord \Einheit\hbox to0pt{\hskip\xcoord \Einheit
         \unitlength\Einheit
         \line(1,-1){#3}\hss}}
\def\Pfad(#1,#2),#3\endPfad{\unskip\leavevmode
  \xcoord#1 \ycoord#2 \thicklines\ZeichnePfad#3\endPfad\thinlines}
\def\ZeichnePfad#1{\ifx#1\endPfad\let\next\relax
  \else\let\next\ZeichnePfad
    \ifnum#1=1
      \raise\ycoord \Einheit\hbox to0pt{\hskip\xcoord \Einheit
         \vrule height\Pfadd@cke width1 \Einheit depth\Pfadd@cke\hss}%
      \advance\xcoord by 1
    \else\ifnum#1=2
      \raise\ycoord \Einheit\hbox to0pt{\hskip\xcoord \Einheit
        \hbox{\hskip-\PfadD@cke\vrule height1 \Einheit width\PfadD@cke depth0pt}\hss}%
      \advance\ycoord by 1
    \else\ifnum#1=3
      \raise\ycoord \Einheit\hbox to0pt{\hskip\xcoord \Einheit
         \unitlength\Einheit
         \line(1,1){1}\hss}
      \advance\xcoord by 1
      \advance\ycoord by 1
    \else\ifnum#1=4
      \raise\ycoord \Einheit\hbox to0pt{\hskip\xcoord \Einheit
         \unitlength\Einheit
         \line(1,-1){1}\hss}
      \advance\xcoord by 1
      \advance\ycoord by -1
    \else\ifnum#1=5
      \advance\xcoord by -1
      \raise\ycoord \Einheit\hbox to0pt{\hskip\xcoord \Einheit
         \vrule height\Pfadd@cke width1 \Einheit depth\Pfadd@cke\hss}%
    \else\ifnum#1=6
      \advance\ycoord by -1
      \raise\ycoord \Einheit\hbox to0pt{\hskip\xcoord \Einheit
        \hbox{\hskip-\PfadD@cke\vrule height1 \Einheit width\PfadD@cke depth0pt}\hss}%
    \else\ifnum#1=7
      \advance\xcoord by -1
      \advance\ycoord by -1
      \raise\ycoord \Einheit\hbox to0pt{\hskip\xcoord \Einheit
         \unitlength\Einheit
         \line(1,1){1}\hss}
    \else\ifnum#1=8
      \advance\xcoord by -1
      \advance\ycoord by +1
      \raise\ycoord \Einheit\hbox to0pt{\hskip\xcoord \Einheit
         \unitlength\Einheit
         \line(1,-1){1}\hss}
    \fi\fi\fi\fi
    \fi\fi\fi\fi
  \fi\next}
\def\hSSchritt{\leavevmode\raise-.4pt\hbox to0pt{\hss.\hss}\hskip.2\Einheit
  \raise-.4pt\hbox to0pt{\hss.\hss}\hskip.2\Einheit
  \raise-.4pt\hbox to0pt{\hss.\hss}\hskip.2\Einheit
  \raise-.4pt\hbox to0pt{\hss.\hss}\hskip.2\Einheit
  \raise-.4pt\hbox to0pt{\hss.\hss}\hskip.2\Einheit}
\def\vSSchritt{\vbox{\baselineskip.2\Einheit\lineskiplimit0pt
\hbox{.}\hbox{.}\hbox{.}\hbox{.}\hbox{.}}}
\def\DSSchritt{\leavevmode\raise-.4pt\hbox to0pt{%
  \hbox to0pt{\hss.\hss}\hskip.2\Einheit
  \raise.2\Einheit\hbox to0pt{\hss.\hss}\hskip.2\Einheit
  \raise.4\Einheit\hbox to0pt{\hss.\hss}\hskip.2\Einheit
  \raise.6\Einheit\hbox to0pt{\hss.\hss}\hskip.2\Einheit
  \raise.8\Einheit\hbox to0pt{\hss.\hss}\hss}}
\def\dSSchritt{\leavevmode\raise-.4pt\hbox to0pt{%
  \hbox to0pt{\hss.\hss}\hskip.2\Einheit
  \raise-.2\Einheit\hbox to0pt{\hss.\hss}\hskip.2\Einheit
  \raise-.4\Einheit\hbox to0pt{\hss.\hss}\hskip.2\Einheit
  \raise-.6\Einheit\hbox to0pt{\hss.\hss}\hskip.2\Einheit
  \raise-.8\Einheit\hbox to0pt{\hss.\hss}\hss}}
\def\SPfad(#1,#2),#3\endSPfad{\unskip\leavevmode
  \xcoord#1 \ycoord#2 \ZeichneSPfad#3\endSPfad}
\def\ZeichneSPfad#1{\ifx#1\endSPfad\let\next\relax
  \else\let\next\ZeichneSPfad
    \ifnum#1=1
      \raise\ycoord \Einheit\hbox to0pt{\hskip\xcoord \Einheit
         \hSSchritt\hss}%
      \advance\xcoord by 1
    \else\ifnum#1=2
      \raise\ycoord \Einheit\hbox to0pt{\hskip\xcoord \Einheit
        \hbox{\hskip-2pt \vSSchritt}\hss}%
      \advance\ycoord by 1
    \else\ifnum#1=3
      \raise\ycoord \Einheit\hbox to0pt{\hskip\xcoord \Einheit
         \DSSchritt\hss}
      \advance\xcoord by 1
      \advance\ycoord by 1
    \else\ifnum#1=4
      \raise\ycoord \Einheit\hbox to0pt{\hskip\xcoord \Einheit
         \dSSchritt\hss}
      \advance\xcoord by 1
      \advance\ycoord by -1
    \else\ifnum#1=5
      \advance\xcoord by -1
      \raise\ycoord \Einheit\hbox to0pt{\hskip\xcoord \Einheit
         \hSSchritt\hss}%
    \else\ifnum#1=6
      \advance\ycoord by -1
      \raise\ycoord \Einheit\hbox to0pt{\hskip\xcoord \Einheit
        \hbox{\hskip-2pt \vSSchritt}\hss}%
    \else\ifnum#1=7
      \advance\xcoord by -1
      \advance\ycoord by -1
      \raise\ycoord \Einheit\hbox to0pt{\hskip\xcoord \Einheit
         \DSSchritt\hss}
    \else\ifnum#1=8
      \advance\xcoord by -1
      \advance\ycoord by 1
      \raise\ycoord \Einheit\hbox to0pt{\hskip\xcoord \Einheit
         \dSSchritt\hss}
    \fi\fi\fi\fi
    \fi\fi\fi\fi
  \fi\next}
\def\Koordinatenachsen(#1,#2){\unskip
 \hbox to0pt{\hskip-.5pt\vrule height#2 \Einheit width.5pt depth1 \Einheit}%
 \hbox to0pt{\hskip-1 \Einheit \xcoord#1 \advance\xcoord by1
    \vrule height0.25pt width\xcoord \Einheit depth0.25pt\hss}}
\def\Koordinatenachsen(#1,#2)(#3,#4){\unskip
 \hbox to0pt{\hskip-.5pt \ycoord-#4 \advance\ycoord by1
    \vrule height#2 \Einheit width.5pt depth\ycoord \Einheit}%
 \hbox to0pt{\hskip-1 \Einheit \hskip#3\Einheit 
    \xcoord#1 \advance\xcoord by1 \advance\xcoord by-#3 
    \vrule height0.25pt width\xcoord \Einheit depth0.25pt\hss}}
\def\Gitter(#1,#2){\unskip \xcoord0 \ycoord0 \leavevmode
  \LOOP\ifnum\ycoord<#2
    \loop\ifnum\xcoord<#1
      \raise\ycoord \Einheit\hbox to0pt{\hskip\xcoord \Einheit\Punkt\hss}%
      \advance\xcoord by1
    \repeat
    \xcoord0
    \advance\ycoord by1
  \REPEAT}
\def\Gitter(#1,#2)(#3,#4){\unskip \xcoord#3 \ycoord#4 \leavevmode
  \LOOP\ifnum\ycoord<#2
    \loop\ifnum\xcoord<#1
      \raise\ycoord \Einheit\hbox to0pt{\hskip\xcoord \Einheit\Punkt\hss}%
      \advance\xcoord by1
    \repeat
    \xcoord#3
    \advance\ycoord by1
  \REPEAT}
\def\Label#1#2(#3,#4){\unskip \xdim#3 \Einheit \ydim#4 \Einheit
  \def\lo{\advance\xdim by-.5 \Einheit \advance\ydim by.5 \Einheit}%
  \def\llo{\advance\xdim by-.25cm \advance\ydim by.5 \Einheit}%
  \def\loo{\advance\xdim by-.5 \Einheit \advance\ydim by.25cm}%
  \def\o{\advance\ydim by.25cm}%
  \def\ro{\advance\xdim by.5 \Einheit \advance\ydim by.5 \Einheit}%
  \def\rro{\advance\xdim by.25cm \advance\ydim by.5 \Einheit}%
  \def\roo{\advance\xdim by.5 \Einheit \advance\ydim by.25cm}%
  \def\l{\advance\xdim by-.30cm}%
  \def\r{\advance\xdim by.30cm}%
  \def\lu{\advance\xdim by-.5 \Einheit \advance\ydim by-.6 \Einheit}%
  \def\llu{\advance\xdim by-.25cm \advance\ydim by-.6 \Einheit}%
  \def\luu{\advance\xdim by-.5 \Einheit \advance\ydim by-.30cm}%
  \def\u{\advance\ydim by-.30cm}%
  \def\ru{\advance\xdim by.5 \Einheit \advance\ydim by-.6 \Einheit}%
  \def\rru{\advance\xdim by.25cm \advance\ydim by-.6 \Einheit}%
  \def\ruu{\advance\xdim by.5 \Einheit \advance\ydim by-.30cm}%
  #1\raise\ydim\hbox to0pt{\hskip\xdim
     \vbox to0pt{\vss\hbox to0pt{\hss$#2$\hss}\vss}\hss}%
}
\catcode`\@=12

\begin{document}

\title[Number of ``$udu$'' of a Dyck path and ad-nilpotent ideals]{Number of ``$udu$'' of a Dyck path and ad-nilpotent ideals of parabolic subalgebras of $sl_{l+1}(\cset)$}
\author{C\'eline RIGHI}
\address{UMR 6086 CNRS, D\'epartement de Math\'ematiques, T\'el\'eport 2 - BP
  30179, Boulevard Marie et Pierre Curie, 86962 Futuroscope
  Chasseneuil Cedex, France}
\email{celine.righi@math.univ-poitiers.fr}

\begin{abstract}
For an ad-nilpotent ideal $\ilie $ of a Borel subalgebra of $sl_{l+1}(\cset)$, we denote by 
$I_{\ilie}$ the maximal subset $I$ of the set of simple roots such that $\ilie$ is an ad-nilpotent ideal of the 
standard parabolic subalgebra $\plie_I$. We use the bijection of \cite{AKOP} between the set of ad-nilpotent ideals of a Borel subalgebra in $sl_{l+1}(\cset)$ and the set of Dyck paths of length $2l+2$, to explicit a bijection between ad-nilpotent ideals $\ilie$ of the 
Borel subalgebra such that $\sharp I_{\ilie}=r$ and the Dyck paths of length $2l+2$ having $r$ occurence ``$udu$''. We obtain 
also a duality between antichains of cardinality $p$ and $l-p$ in the set of positive roots.
\end{abstract}


\maketitle

\section{Introduction}
Let $M_{l+1}(\cset)$ be the set of $(l+1)$-by-$(l+1)$ matrices with coefficients in $\cset$, and $\glie$ be the simple Lie algebra $sl_{l+1}(\cset)$ consisting of elements of $M_{l+1}(\cset)$ whose trace is equal to zero. 
Let $\hlie$ be the maximal torus of $\glie$ consisting of trace zero diagonal matrices. Let $(E_{i,j})$ be the canonical basis 
of $M_{l+1}(\cset)$ and $(E_{i,j}^*)$ be its dual basis. For $1\leqslant i \leqslant l+1$, set $\epsilon_i=E_{i,i}^*$. Then $\Delta=\{\epsilon_i -\epsilon_j; 1\leqslant i,j\leqslant l+1, i\not= j\}$ is the root system associated to $(\glie,\hlie)$, and  $\Delta^+=\{\epsilon_i -\epsilon_j; 1\leqslant i<j\leqslant l+1\}$ is a system of positive roots. Denote by $\alpha_i=\epsilon_i-\epsilon_{i+1}$, for $i=1,\dots, l$. Then $\Pi=\{\alpha_1,
\dots, \alpha_l\}$ is the corresponding set of simple roots. For each $\alpha\in\Delta$, let 
$ \mathfrak{g}_{ \alpha}=\{x\in\glie; [h,x]=\alpha(h)x \mbox{ for all } h\in\hlie\}$ be the root space of $\glie$ relative to $\alpha$.

For  $I \subset \Pi$, set $\Delta_I=\zset I \cap\Delta$.
We fix the corresponding standard parabolic subalgebra :   
$$
\plie_I=\hlie\oplus\left(\bigoplus\limits_{\alpha \in \Delta_I \cup \Delta^+}\glie_{\alpha}\right).
$$

An ideal $\ilie $ of $\plie_I$ is ad-nilpotent if and only if for all $x\in \ilie$, $ad_{\plie_I} x$ is nilpotent. Since any ideal of $\plie_I$ is $\hlie$-stable, we can deduce easily that an ideal is ad-nilpotent if and only if it is nilpotent. Moreover, we have $\ilie =\bigoplus\limits_{\alpha \in \Phi} \glie_{\alpha}$, for some subset $\Phi\subset \Delta^+ \setminus \Delta_I$. 

A Dyck path of length $2n$ can be defined as a word of $2n$ letters $u$ or $d$, having the same number of $u$ and $d$, and such that there is always more $u$'s than $d$'s to the left of a letter.


Andrews, Krattenthaler, Orsina and Papi established in \cite{AKOP} a bijection between the set of ad-nilpotent ideals of the Borel 
subalgebra $\plie_{\emptyset}$ and the set of Dyck paths of length $2l+2$ which allows them to enumerate ad-nilpotent ideals of a 
fixed class of nilpotence. The purpose of this paper is to explain some applications of this correspondence for the ad-nilpotent ideals of parabolic subalgebras.

More precisely, let $\ilie$ be an ad-nilpotent ideal of the Borel subalgebra $\plie_{\emptyset}$. Denote by 
$I_{\ilie}$ the maximal subset $I\subset \Pi$ such that $\ilie$ is an ad-nilpotent ideal of $\plie_I$, the main result we prove here is the following theorem :

\begin{thm*}
There is a bijection between the ad-nilpotent ideals $\ilie$ of $\plie_{\emptyset}$ such that $\sharp I_{\ilie}=r$ and the Dyck paths of length $2l+2$ having $r$ occurence ``$udu$''.

\end{thm*}

We then deduce a formula for the number of ad-nilpotent ideals of $\plie_{\emptyset}$ such that the cardinality of $I_{\ilie}$ is equal to $r$.

This paper is organized as follows : we first recall the natural bijection between $l$-partitions and Dyck paths of length $2l+2$, as in \cite{P}. In section 3, we recall the iterative construction of the bijection of \cite{AKOP}. Then, in section 4, we explain how to calculate the number of occurence ``$udu$'' of a Dyck path obtained by the previous construction. In section 5, we recall some facts of \cite{R} and \cite{CP} on ad-nilpotent ideals and we prove Theorem 1. Finally, in section 6, we establish a duality between ad-nilpotent ideals of $\plie_{\emptyset}$. Such a duality has already been constructed by Panyushev in \cite{P}, however, it is not the same as the one we have here.

\bigskip
\noindent
{\bf Acknowledgment.} This work was realized while I was visiting the Istituto Guido Castelnuvo di 
Matematica (Roma). I would like to thank the european program Liegrits for offering me the possibility to go there and the institute for its hospitality.

\section{Partitions and Dyck paths}\label{partition_Dyck_path}
 
In this section, we shall see how to generate a Dyck path from a partition. 

Recall that a partition is an $l$-tuple $\lambda=(\lambda_1, \lambda_2,\dots, \lambda_l)\in\nset^l$ such that $\lambda_1 \geqslant \lambda_2 \geqslant \dots\geqslant \lambda_l$. A partition will be called an $l$-partition if $\lambda_i\leqslant i$ for $i=1,\dots, l$.  

Partitions are usually represented by their Ferrers diagrams. Let $T_l$ be the Ferrers diagram of the $l$-partition $(l,l-1,\dots,1)$. Then the Ferrers diagram $F$ of any $l$-partition $\lambda$ can be viewed as a subdiagram of $T_l$. For example, for $l=5$, the Ferrers diagram of 
$\lambda=(3,1,1,0,0)$ is the subdiagram of $T_l$, whose boxes are denoted by some $\star$ :
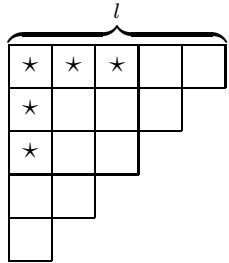
$$
\overbrace{
\vbox{\offinterlineskip
\halign{#&#&#&#&#&#\cr
\htrait&\htrait&\htrait&\htrait&\htrait\cr
\sboite{$\star$}&\sboite{$\star$}&\sboite{$\star$}&\sboite{}&%
\sboite{}&\vtrait\cr
\htrait&\htrait&\htrait&\htrait&\htrait\cr
\sboite{$\star$}&\sboite{}&\sboite{}%
&\sboite{}&\vtrait\cr
\htrait&\htrait&\htrait&\htrait\cr
\sboite{$\star$}&\sboite{}&\sboite{}&\vtrait\cr
\htrait&\htrait&\htrait\cr
\sboite{}&\sboite{}&\vtrait\cr
\htrait&\htrait\cr
\sboite{}&\vtrait\cr
\htrait\cr
}
}
}^{l}
$$

Let $\lambda=(\lambda_1,\dots,\lambda_l)$ be an $l$-partition and let $F$ be its Ferrers diagram. We draw a dotted horizontal line from the top of the line $x+y=l+1$ to $F$ and a dotted vertical line from $F$ to the bottom of the line $x+y=l+1$. For example, when $\lambda=(5,3,1,1,1,0,0)$, we have :
$$
\begin{Tiles}{1}
\Dbloc{\Ttop\Tleft}\Dbloc{\Ttop}\Dbloc{\Ttop}\Dbloc{\Ttop\Tbottom}\Dbloc{\Ttop\Tbottom\Tright}\Dbloc{\Ttopleftdot\Ttopdots}\Dbloc{\Ttopleftdot\Ttopdots}\Dbloc{\Ttopleftdot\Ttopdots\Ttoprightdot\Tantidiagonal}\Dskip
\Dbloc{\Tleft}\Dbloc{\Tbottom}\Dbloc{\Tright\Tbottom}\Dspace\Dspace\Dbloc{}\Dbloc{\Tantidiagonal}\Dbloc{\Ddoubletext{c}{x+y=l+1}}\Dskip
\Dbloc{\Tleft\Tright}\Dbloc{}\Dbloc{}\Dbloc{}
\Dbloc{}\Dbloc{\Tantidiagonal}\Dskip
\Dbloc{\Tleft\Tright}\Dspace\Dspace\Dspace\Dbloc{\Tantidiagonal}\Dskip
\Dbloc{\Tleft\Tright\Tbottom}\Dspace\Dspace\Dbloc{\Tantidiagonal}\Dskip
\Dbloc{\Ttopleftdot\Tleftdots}\Dspace\Dbloc{\Tantidiagonal}\Dskip
\Dbloc{\Ttopleftdot\Tleftdots}\Dbloc{\Tantidiagonal}\Dskip
\Dbloc{\Ttopleftdot\Tleftdots\Tbottomrightdot\Tantidiagonal}\Dskip
\end{Tiles}
$$
\caption{}\label{fig_classique}

If we rotate the figure clockwise by $45$ degrees, we can easily see that we obtain  a Dyck path of length $2l+2$ called $P(\lambda)$ as in \cite{P}. This construction defines clearly a bijection $P : \lambda \mapsto P(\lambda)$ between $l$-partitions and Dyck paths of length $2l+2$. In the above example, the Dyck path $P(\lambda)$ is :
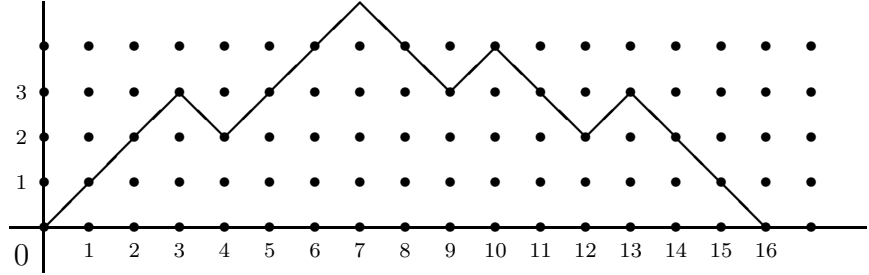
\begin{center}
$$
\Gitter(18,5)(0,0)
\Koordinatenachsen(18
,5)(0,0)
\Pfad(0,0),3334333443443444\endPfad
\Label\lu{0}(0,0)
\Label\u{\scriptstyle 1}(1,0)
\Label\u{\scriptstyle 2}(2,0)
\Label\u{\scriptstyle 3}(3,0)
\Label\u{\scriptstyle 4}(4,0)
\Label\u{\scriptstyle 5}(5,0)
\Label\u{\scriptstyle 6}(6,0)
\Label\u{\scriptstyle 7}(7,0)
\Label\u{\scriptstyle 8}(8,0)
\Label\u{\scriptstyle 9}(9,0)
\Label\u{\scriptstyle 10}(10,0)
\Label\u{\scriptstyle 11}(11,0)
\Label\u{\scriptstyle 12}(12,0)
\Label\u{\scriptstyle 13}(13,0)
\Label\u{\scriptstyle 14}(14,0)
\Label\u{\scriptstyle 15}(15,0)
\Label\u{\scriptstyle 16}(16,0)
\Label\l{\scriptstyle 1}(0,1)
\Label\l{\scriptstyle 2}(0,2)
\Label\l{\scriptstyle 3}(0,3)
\hskip10.5cm
$$
\end{center}


\section{AKOP-bijection}\label{section_AKOP_bijection}

Let $\lambda=(\lambda_1,\dots,\lambda_l)$ be a $l$-partition whose Ferrers diagram is $F$. We shall draw a dotted line associated to $\lambda$. We start at the top of the line $x+y=l+1$. We go left until we meet $F$. Then, we continue downwards until we reach $x+y=l+1$. Then we iterate the procedure  until we reach the bottom. 
For example, for $l=13$ and $\lambda=(10,10,9,6,5,4,4,3,1,1,1,1,0)$ :

\begin{center}
$$
\begin{Tiles}{0.9}
\Dbloc{\Ttop\Tleft}\Dbloc{\Ttop}\Dbloc{\Ttop}\Dbloc{\Ttop}\Dbloc{\Ttop}\Dbloc{\Ttop}\Dbloc{\Ttop}\Dbloc{\Ttop}\Dbloc{\Ttop}\Dbloc{\Ttop\Tright}\Dbloc{\Ttopleftdot\Ttopdots}\Dbloc{\Ttopleftdot\Ttopdots}\Dbloc{\Ttopleftdot\Ttopdots}\Dbloc{\Ttopleftdot\Ttopdots\Tantidiagonal}\Dskip
\Dbloc{\Tleft}\Dspace\Dspace\Dspace\Dspace\Dbloc{}\Dspace\Dspace\Dspace\Dbloc{\Tright\Tbottom}\Dspace\Dspace\Dbloc{\Tantidiagonal}\Dbloc{\Ddoubletext{c}{x+y=l+1}}\Dskip
\Dbloc{\Tleft}\Dspace\Dspace\Dspace\Dspace\Dbloc{}\Dbloc{\Tbottom}\Dbloc{\Tbottom}\Dbloc{\Tbottom\Tright}\Dbloc{\Trightdots\Tbottomrightdot}\Dspace\Dbloc{\Tantidiagonal}\Dskip
\Dbloc{\Tleft}\Dspace\Dspace\Dspace\Dspace\Dbloc{\Tbottom\Tright}\Dbloc{\Tbottomleftdot\Tbottomdots}\Dbloc{\Tbottomleftdot\Tbottomdots}\Dbloc{\Tbottomleftdot\Tbottomdots}\Dbloc{\Tbottomleftdot\Tbottomdots\Trightdots\Tbottomrightdot}\Dbloc{\Tantidiagonal\Tbottomleftdot}\Dskip
\Dbloc{\Tleft}\Dbloc{}\Dbloc{}\Dbloc{}\Dbloc{\Tbottom}\Dbloc{\Tleft}\Dspace\Dspace\Dspace\Dbloc{\Tantidiagonal}\Dskip
\Dbloc{\Tleft}\Dspace\Dspace\Dspace\Dbloc{\Tleft\Trightdots\Tbottomrightdot}\Dspace\Dspace\Dspace\Dbloc{\Tantidiagonal}\Dskip
\Dbloc{\Tleft}\Dspace\Dspace\Dspace\Dbloc{\Tleft\Trightdots\Tbottomrightdot}\Dspace\Dspace\Dbloc{\Tantidiagonal}\Dskip
\Dbloc{\Tleft}\Dbloc{\Tbottom}\Dbloc{\Tbottom\Tright}\Dbloc{\Ttop}\Dspace\Dbloc{\Tleftdots\Tbottomleftdot}\Dbloc{\Tantidiagonal}\Dskip
\Dbloc{\Tleft}\Dbloc{\Tleft\Tbottomleftdot\Tbottomdots}\Dbloc{\Tbottomleftdot\Tbottomdots}\Dbloc{\Tbottomleftdot\Tbottomdots}\Dbloc{\Tbottomleftdot\Tbottomdots\Tbottomrightdot}\Dbloc{\Tleftdots\Tantidiagonal}\Dskip
\Dbloc{\Tleft\Tright}\Dspace\Dspace\Dspace\Dbloc{\Tantidiagonal}\Dskip
\Dbloc{\Tleft\Tright}\Dspace\Dspace\Dbloc{\Tantidiagonal}\Dskip
\Dbloc{\Tleft\Tbottom\Tright}\Dspace\Dbloc{\Tantidiagonal}\Dskip
\Dbloc{\Tbottomleftdot\Trightdots\Tbottomrightdot\Tbottomdots}\Dbloc{\Tantidiagonal}\Dskip
\Dbloc{\Tleftdots\Tantidiagonal}
\end{Tiles}
$$
\caption{}{\label{fig1}}
\end{center}

Let $n(\lambda)$ be the number of points of the dotted line on $x+y=l+1$, which are not at the top or bottom. For example, we have $n((0,\dots,0))=0$, and for the $l$-partition $\lambda$ of Figure \ref{fig1}, we have $n(\lambda)=3$.

We shall describe the construction of this line in a more formal way. 

Let $k=n(\lambda)$. Set $i_n=l+1$ for all $n>k$, $i_k=\lambda_1$, $i_{k-1}=\lambda_{l-i_k+2}$, $i_{k-2}=\lambda_{l-i_{k-1}+2}$, $\dots$, $i_{1}=\lambda_{l-i_2+2}$ and $i_p=0$ for all $p\leqslant 0$. We have $0<i_1<\dots <i_k<l+1$. The dotted line describes the shape of an $l$-partition 
\begin{equation}\label{partmax}
\lambda^M=(i_k^{l-i_k+1}, i_{k-1}^{i_k-i_{k-1}},\dots, i_1^{i_2-i_1}, 0^{i_1-1}).
\end{equation}

Any $l$-partition $\lambda$ whose associated dotted line gives the partition  $\lambda^M$ must necessarily contain the cells 
$$
(1,i_k),(l-i_k+2,i_{k-1}),(l-i_{k-1}+2,i_{k-2}),\dots, (l-i_2+2,i_1).
$$ 
The ``minimal'' $l$-partition in the sense of inclusion of diagrams that contains these cells is : 
\begin{equation}\label{partmin}
\lambda^m=(i_k, i_{k-1}^{l-i_k+1},i_{k-2}^{i_k-i_{k-1}},\dots, i_1^{i_3-i_2}, 0^{i_2-2}).
\end{equation}

For example, take $l=13$ and $\lambda=(10,10,9,6,5,4,4,3,1,1,1,1,0)$, as above, we have $n(\lambda)=k=3$, $i_3=10,i_2=5,i_1=1$. The three distinguished cells above are :  
$$
(1,10), (5,5), (10,1).
$$
So we have :
$$
\begin{array}{c}
\lambda^M=(10,10,10,10, 5,5,5,5,5,1,1,1,1), \mbox{ and} \\
\lambda^m=(10,5,5,5,5,1,1,1,1,1,0,0,0).
\end{array}
$$
These partitions are illustrated in the figure below, where the distinguished cells are marked with $\times$, and $\lambda^M$ is the partition corresponding to the dotted line outside $\lambda$, while $\lambda^m$ is the one which corresponds to the dotted line inside $\lambda$. 
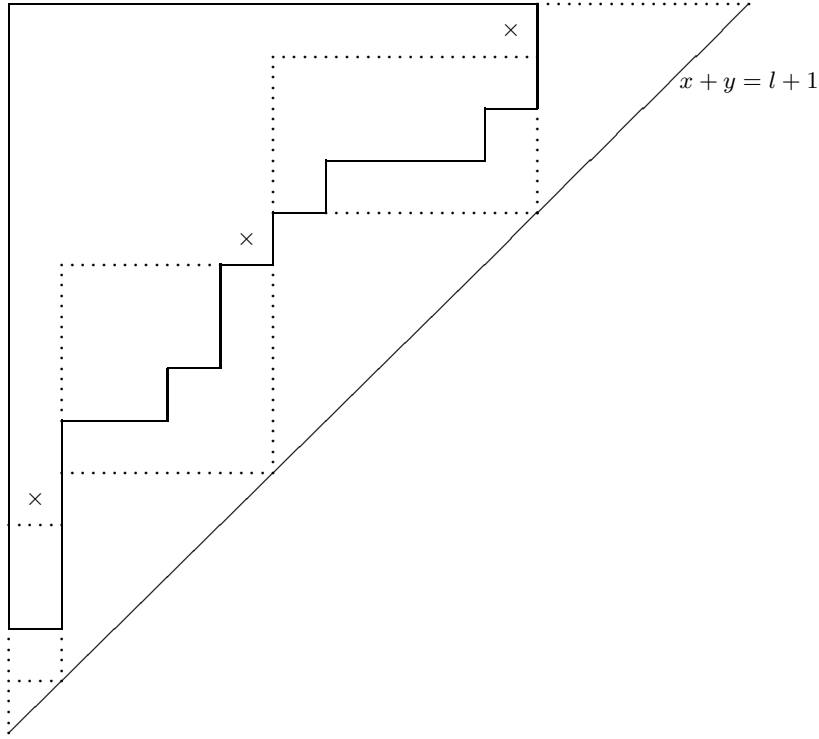
$$
\begin{Tiles}{1}\label{fig2}
\Dbloc{\Ttop\Tleft}\Dbloc{\Ttop}\Dbloc{\Ttop}\Dbloc{\Ttop}\Dbloc{\Ttop}\Dbloc{\Ttop\Tbottomleftdot\Tbottomdots}\Dbloc{\Ttop\Tbottomleftdot\Tbottomdots}\Dbloc{\Ttop\Tbottomleftdot\Tbottomdots}\Dbloc{\Ttop\Tbottomleftdot\Tbottomdots}\Dbloc{\Ttop\Tright\Tbottomleftdot\Tbottomrightdot\Tbottomdots\Dtext{c}{\times}}\Dbloc{\Ttopleftdot\Ttopdots}\Dbloc{\Ttopleftdot\Ttopdots}\Dbloc{\Ttopleftdot\Ttopdots}\Dbloc{\Ttopleftdot\Ttopdots\Ttoprightdot\Tantidiagonal}\Dskip
\Dbloc{\Tleft}\Dspace\Dspace\Dspace\Dspace\Dbloc{\Tleftdots\Tbottomleftdot}\Dspace\Dspace\Dspace\Dbloc{\Tright\Tbottom}\Dspace\Dspace\Dbloc{\Tantidiagonal}\Dbloc{\Ddoubletext{c}{x+y=l+1}}\Dskip
\Dbloc{\Tleft}\Dspace\Dspace\Dspace\Dspace\Dbloc{\Tleftdots\Tbottomleftdot}\Dbloc{\Tbottom}\Dbloc{\Tbottom}\Dbloc{\Tbottom\Tright}\Dbloc{\Trightdots\Tbottomrightdot}\Dspace\Dbloc{\Tantidiagonal}\Dskip
\Dbloc{\Tleft}\Dspace\Dspace\Dspace\Dspace\Dbloc{\Tleftdots\Tbottomleftdot\Tbottom\Tright}\Dbloc{\Tbottomleftdot\Tbottomdots}\Dbloc{\Tbottomleftdot\Tbottomdots}\Dbloc{\Tbottomleftdot\Tbottomdots}\Dbloc{\Tbottomleftdot\Tbottomdots\Trightdots\Tbottomrightdot}\Dbloc{\Tantidiagonal\Tbottomleftdot}\Dskip
\Dbloc{\Tleft}\Dbloc{\Tbottomleftdot\Tbottomdots}\Dbloc{\Tbottomleftdot\Tbottomdots}\Dbloc{\Tbottomleftdot\Tbottomdots\Tbottomrightdot}\Dbloc{\Tbottom\Dtext{c}{\times}}\Dbloc{\Tleft}\Dspace\Dspace\Dspace\Dbloc{\Tantidiagonal}\Dskip
\Dbloc{\Tleft\Trightdots\Tbottomrightdot}\Dspace\Dspace\Dspace\Dbloc{\Tleft\Trightdots\Tbottomrightdot}\Dspace\Dspace\Dspace\Dbloc{\Tantidiagonal}\Dskip
\Dbloc{\Tleft\Trightdots\Tbottomrightdot}\Dspace\Dspace\Dspace\Dbloc{\Tleft\Trightdots\Tbottomrightdot}\Dspace\Dspace\Dbloc{\Tantidiagonal}\Dskip
\Dbloc{\Tleft\Trightdots\Tbottomrightdot}\Dbloc{\Tbottom}\Dbloc{\Tbottom\Tright}\Dbloc{\Ttop}\Dspace\Dbloc{\Tleftdots\Tbottomleftdot}\Dbloc{\Tantidiagonal}\Dskip
\Dbloc{\Tleft}\Dbloc{\Tleft\Tbottomleftdot\Tbottomdots}\Dbloc{\Tbottomleftdot\Tbottomdots}\Dbloc{\Tbottomleftdot\Tbottomdots}\Dbloc{\Tbottomleftdot\Tbottomdots\Tbottomrightdot}\Dbloc{\Tleftdots\Tantidiagonal}\Dskip
\Dbloc{\Tleft\Tbottomleftdot\Tbottomdots\Tbottomrightdot\Tright\Dtext{c}{\times}}\Dspace\Dspace\Dspace\Dbloc{\Tantidiagonal}\Dskip
\Dbloc{\Tleft\Tright}\Dspace\Dspace\Dbloc{\Tantidiagonal}\Dskip
\Dbloc{\Tleft\Tbottom\Tright}\Dspace\Dbloc{\Tantidiagonal}\Dskip
\Dbloc{\Tleftdots\Tbottomleftdot\Trightdots\Tbottomrightdot\Tbottomdots}\Dbloc{\Tantidiagonal}\Dskip
\Dbloc{\Tleftdots\Tbottomleftdot\Tantidiagonal}
\end{Tiles}
$$

Observe that the difference $\lambda^M\setminus\lambda^m$ is a disjoint union of $k$ rectangles, denoted by $R_k,\dots,R_1$ from the top to the bottom. More precisely,
$$
R_j=\{(s,t);l-i_{p+1}+2<s<l-i_{p}+2 \mbox{ and } i_{p-1} <t \leqslant i_p\}.
$$
Inside each rectangle $R_j$, the shape of $\lambda$ could be described by a word $M_j$, whose letters are $d$ and $l$, where $d$ indicates a down step and $l$ indicates a left step. 

Let $h_j$ be the number of $d$ in $M_j$, which is at most the height of $R_j$ and let $l_j$ be the number of $l$ in $M_j$, which is the length of $R_j$. Then we have :
$$
\begin{array}{c}
h_j=i_{j+1}-i_j-1 \mbox{ if } j\not=1, \mbox{ and } h_j\leqslant i_{j+1}-i_j-1 \mbox{ if } j=1, \\
l_j=i_j-i_{j-1},
\end{array}
$$
so $h_j\leqslant l_{j+1}-1$ and the equality holds if $j\not= 1$.
Furthermore the shape of $M_j$ is $l^{a_{j,0}}dl^{a_{j,1}}d \dots dl^{a_{j,h_j}}$, where $a_{j,i}\in\nset$, $0\leqslant i\leqslant h_j$. We then have that :
\begin{equation}\label{equation_a_i}
l_j=\sum_{i=0}^{h_j} a_{j,i}.
\end{equation}
In the above example, we have $M_3=dldl^3dl$, $M_2=lddldl^2d$ and $M_1=ddl$.

We shall now generate a Dyck path step by step from the $M_j$. We call a peak of a Dyck path, an occurrence of $ud$ in the corresponding Dyck word. 

First, let $D_{k+1}$ be the Dyck path of length $2(l+1-i_k)$ containing $l+1-i_k$ peaks. Next, we have $M_k=l^{a_{k,0}}dl^{a_{k,1}}d \dots dl^{a_{k,h_k}}$. We insert $a_{k,0}$ peaks on the first peak of the already existing Dyck path $D_{k+1}$, then $a_{k,1}$ peaks on the second peak, and so on. We call $D_k$ the new Dyck path obtained. Observed that the highest peaks of $D_k$ are exactly those newly inserted, so there are exactly $l_k$. Since $h_{k-1}\leqslant l_k-1$, the procedure can then be iterated by inserting peaks only on highest peaks. Each intermediate Dyck path obtained after using the word $M_j$ is denoted by $D_{j}$. At the end, we obtain a Dyck path $D_{\lambda}$ of length $2l+2$.

For example, if we consider $l=7$ and $\lambda=(5,3,1,1,1,0,0)$ :
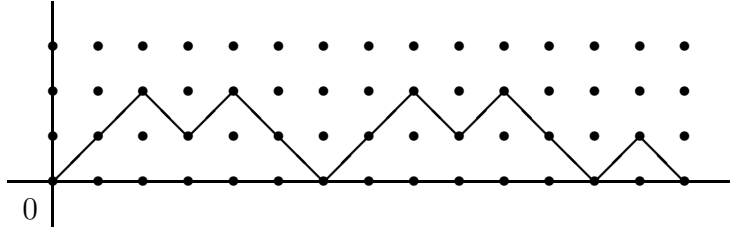
$$
\begin{Tiles}{1}
\Dbloc{\Ttop\Tleft}\Dbloc{\Ttop\Tbottomleftdot\Tbottomdots}\Dbloc{\Ttop\Tbottomleftdot\Tbottomdots}\Dbloc{\Ttop\Tbottom}\Dbloc{\Ttop\Tbottom\Tright}\Dbloc{}\Dbloc{}\Dbloc{\Tantidiagonal}\Dskip
\Dbloc{\Tleft\Tbottomrightdot\Trightdots}\Dbloc{\Tbottom}\Dbloc{\Tright\Tbottom}\Dspace\Dspace\Dbloc{\Tleftdots\Tbottomleftdot}\Dbloc{\Tantidiagonal}\Dbloc{\Ddoubletext{c}{x+y=l+1}}\Dskip
\Dbloc{\Tleft\Tright}\Dbloc{\Tbottomleftdot\Tbottomdots}\Dbloc{\Tbottomleftdot\Tbottomdots}\Dbloc{\Tbottomleftdot\Tbottomdots}
\Dbloc{\Tbottomleftdot\Tbottomdots\Tbottomrightdot\Trightdots}\Dbloc{\Tantidiagonal}\Dskip
\Dbloc{\Tleft\Tright\Tbottomdots\Tbottomrightdot\Tbottomleftdot}\Dspace\Dspace\Dspace\Dbloc{\Tantidiagonal}\Dskip
\Dbloc{\Tleft\Tright\Tbottom}\Dspace\Dspace\Dbloc{\Tantidiagonal}\Dskip
\Dbloc{\Tleftdots\Tbottomleftdot\Trightdots\Tbottomrightdot}\Dspace\Dbloc{\Tantidiagonal}\Dskip
\Dbloc{\Tleftdots\Tbottomleftdot\Trightdots\Tbottomrightdot\Tbottomdots}\Dbloc{\Tantidiagonal}\Dskip
\Dbloc{\Tantidiagonal}\Dskip
\end{Tiles}
$$
\caption{}\label{example1}

We have $n(\lambda)=k=2$, $i_2=5$ and $i_1=1$. Then $D_3$ is the following Dyck path : 
\begin{center}
$$
\Gitter(7,3)(0,0)
\Koordinatenachsen(7,2)(0,0)
\Pfad(0,0),343434\endPfad
\Label\lu{0}(0,0)
\hskip4.5cm
$$
\end{center}

We have $M_2=l^2dl^2d$, so we first insert $2$ peaks on the first peak of $D_3$, then again two peaks on the second one. We obtain $D_2$:
\begin{center}
$$
\Gitter(15,4)(0,0)
\Koordinatenachsen(15,4)(0,0)
\Pfad(0,0),33434433434434\endPfad
\Label\lu{0}(0,0)
\hskip7.5cm
$$
\end{center}
Finally, $M_1=dl$ so we insert $a_{1,0}=0$ peak on the first highest peak of $D_2$ and $a_{1,1}=1$ peak on the second highest peak. We obtain $D_{\lambda}$:
\begin{center}
$$
\Gitter(18,5)(0,0)
\Koordinatenachsen(18
,5)(0,0)
\Pfad(0,0),3343344433434434\endPfad
\Label\lu{0}(0,0)
\Label\u{\scriptstyle 1}(1,0)
\Label\u{\scriptstyle 2}(2,0)
\Label\u{\scriptstyle 3}(3,0)
\Label\u{\scriptstyle 4}(4,0)
\Label\u{\scriptstyle 5}(5,0)
\Label\u{\scriptstyle 6}(6,0)
\Label\u{\scriptstyle 7}(7,0)
\Label\u{\scriptstyle 8}(8,0)
\Label\u{\scriptstyle 9}(9,0)
\Label\u{\scriptstyle 10}(10,0)
\Label\u{\scriptstyle 11}(11,0)
\Label\u{\scriptstyle 12}(12,0)
\Label\u{\scriptstyle 13}(13,0)
\Label\u{\scriptstyle 14}(14,0)
\Label\u{\scriptstyle 15}(15,0)
\Label\u{\scriptstyle 16}(16,0)
\Label\l{\scriptstyle 1}(0,1)
\Label\l{\scriptstyle 2}(0,2)
\Label\l{\scriptstyle 3}(0,3)
\hskip10.5cm
$$
\end{center}

By \cite{AKOP}, we have the following proposition :
\begin{prop}\label{map_D}
The map $D : \lambda \mapsto D_{\lambda}$ defines a bijection between the set of $l$-partitions and the set of Dyck paths of length $2l+2$.
\end{prop}

\section{Dyck path and number of ``$udu$''}
Let $\lambda=(\lambda_1, \dots, \lambda_l)$ be an $l$-partition such that $n(\lambda)=k$. Let $D_{\lambda}$ be the Dyck path obtained from $\lambda$ as described in section \ref{section_AKOP_bijection}. We shall see how to count the number of ``$udu$'' contained in $D_{\lambda}$.

A peak could be followed by an $u$, a $d$ or nothing in the Dyck word. If it is followed by a $u$, we call it a $u$-peak. Each $u$-peak will give an $udu$ and vice versa.

Let $1\leqslant j\leqslant k+1$. Let $u_j$ be the number of $u$-peaks in the Dyck path $D_j$. For example, $D_{k+1}$ consists in $l-\lambda_1+1=l-i_k+1$ peaks, so it is easy to see that $u_{k+1}=l-\lambda_1$.

To construct $D_{j-1}$ from $D_j$, we add some peaks on the highest peaks of $D_j$. Then, one must understand how the insertion of $p$ peaks on a highest peak modifies the number of ``$udu$''. Consider a peak $P$ of maximal height on a Dyck path. If we add  $p$ peaks, the part of the Dyck word which corresponds to $P$ (which was $ud$) becomes $uudud\dots udd$ (with $p$ $ud$), so we obtain $p-1$ $udu$. If $P$ is a $u$-peak, then we also ``destroy'' the $udu$ given by $P$. So at the end, we only add $p-2$ $udu$. For example, in the following Dyck path which contains $2$ $udu$ :
\begin{center}
$$
\Gitter(10,3)(0,0)
\Koordinatenachsen(10,3)(0,0)
\Pfad(0,0),34334344\endPfad
\Label\lu{0}(0,0)
\hskip4.5cm
$$
\caption{}\label{Dyckpath}
\end{center}
if we add $2$ peaks on the first highest peak, we add $2-2=0$ $udu$, so we obtain the following Dyck path with still $2$ $udu$:
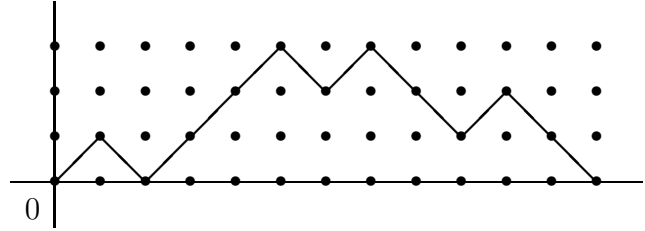
\begin{center}
$$
\Gitter(13,4)(0,0)
\Koordinatenachsen(13,4)(0,0)
\Pfad(0,0),343334344344\endPfad
\Label\lu{0}(0,0)
\hskip5.5cm
$$
\end{center}

If $P$ is not a $u$-peak, then we do not ``destroy'' a $udu$, so we effectly add $p-1$ ``$udu$''. For example, if we add $2$ peaks on the second highest peak of Figure \ref{Dyckpath}, we add $2-1=1$ $udu$, so we obtain $3$ $udu$ at the end:
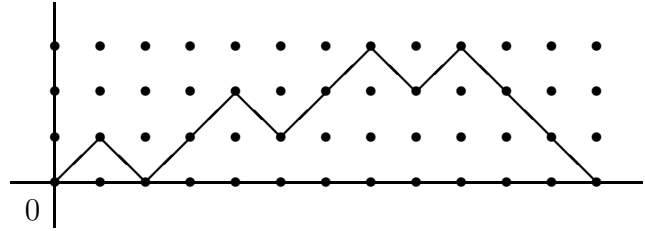
\begin{center}
$$
\Gitter(13,4)(0,0)
\Koordinatenachsen(13,4)(0,0)
\Pfad(0,0),343343343444\endPfad
\Label\lu{0}(0,0)
\hskip5.5cm
$$
\end{center}

Set $a_{k+1,0}=l-i_k+1$, $M_{k+1}=l^{a_{k+1,0}}$, and $h_{k+1}=0$. We have seen that each word $M_j$ is in the form $l^{a_{j,0}}dl^{a_{j,1}}d \dots dl^{a_{j,h_j}}$. Let 
$$
\cala_j=\{(j,t); t\in\{0,\dots,h_j\}; a_{j,t}\not=0\},
$$ 
$$
\cala=\bigcup_{j=1}^k \cala_j.
$$ 
Recall from the construction that the number of highest peaks in $D_j$ is :
\begin{equation}\label{equation_highest_peak}
\sum_{t=0}^{h_j} a_{j,i}=l_j.
\end{equation} 
Observe that a highest peak is a $u$-peak if it is not the last one of a consecutive group of highest peaks. Hence, the $q$-th peak of $D_j$ is not a $u$-peak if and only if there exists $r\in \{0,\dots, h_j\}$ such that $q=\sum_{s=0}^r a_{j,s}$. Set 
$$
\call_p=\left\{
(p,t); \mbox{there exists }\ 0\leqslant r\leqslant h_{p+1}; t+1=\sum_{q=0}^r a_{p+1,q}
\right\},
$$
$$
\begin{array}{ccc}
\calu_p=\cala_p\setminus \call_p, &
\displaystyle\call=\bigcup_{p=1}^k \call_p,& 
\displaystyle\calu=\bigcup_{p=1}^k \calu_p.
\end{array}
$$
Thus $\call_j$ corresponds exactly to the set of highest peaks in $D_j$ which are not $u$-peaks and where we insert new peaks. It follows that :
$$
u_{j-1}=u_j +\sum_{(j-1,t)\in \calu_{j-1}}(a_{j-1,t}-2) +\sum_{(j-1,t)\in \call_{j-1}}(a_{j-1,t}-1).
$$

At the end of the construction, the number of ``$udu$'' in $D_{\lambda}$ is $u_1$. By induction, we have  
$$
u_1= l-\lambda_1 +\displaystyle \sum_{(j,t)\in\calu}(a_{j,t}-2) + \sum_{(j,t)\in\call}(a_{j,t}-1).
$$
Since $\sum_{(j,t)\in\cala}a_{j,t}=\lambda_1$, we obtain the following proposition :
\begin{prop}\label{proposition_partition}
Let $\lambda$ be an $l$-partition. Then, the number of ``$udu$'' in $D_{\lambda}$ is $l-2\sharp \calu-\sharp \call$.
\end{prop}

To illustrate this, we could follow again the construction of the Dyck path which corresponds to $\lambda=(5,3,1,1,1,0,0)$. We first have the Dyck path $D_3$ in Section \ref{section_AKOP_bijection}, with $n-\lambda_1+1=3$ peaks, and $u_3=2$. Then we use the word $M_2=l^2dl^2d=l^{a_{2,0}}dl^{a_{2,1}}d$, where $a_{2,0}, a_{2,1}\in \call_2$, so we add $a_{2,0}-2+a_{2,1}-2=0$ peak. So $u_2=2$. Then we use the word $M_1=dl=l^{a_{1,0}}dl^{a_{1,1}}$, where $a_{1,1}\in \calu_1$, so we add $a_{1,1}-1=0$ peak. Hence, $u_1=2$.

\section{Ad-nilpotent ideals of a parabolic subalgebra and Dyck paths}\label{ideaux_Dyck_path}

Let $I \subset \Pi$ and $\ilie$ be an ad-nilpotent ideal of
$\plie_I$.  We set 
$$
\Phi_{\ilie} =  \{ \alpha \in \Delta^+ \setminus \Delta_I;\ 
\mathfrak{g}_{\alpha } \subseteq \ilie \}.
$$ 
Then $\ilie = 
\bigoplus_{\alpha \in \Phi_{\ilie} } \mathfrak{g}_{\alpha}$ and if 
$\alpha \in \Phi_{\ilie}$, $\beta \in \Delta^+\cup \Delta_I$ are such that $\alpha +\beta
\in \Delta^+$, then $\alpha +\beta \in \Phi_{\ilie}$. 

Conversely, set 
$$ 
\calf_I = \{ \Phi \subset \Delta^+\setminus \Delta_I;\mbox{if } \alpha \in \Phi,
\beta \in \Delta^+\cup \Delta_I, \alpha +\beta \in \Delta^+, 
\mbox{then} \ \alpha +\beta\in \Phi \}.
$$
Then for $\Phi \in \calf_I$, $\ilie_{\Phi} = \bigoplus_{\alpha \in
\Phi} \mathfrak{g}_{\alpha}$ is an ad-nilpotent ideal of $\plie_I$. 

We obtain therefore a bijection
$$
\{\mbox{ad-nilpotent ideals of } \plie_I \}  \rightarrow  \calf_I ,\ 
\ilie  \mapsto  \Phi_{\ilie}.
$$

Recall the following partial order on $\Delta^+$ : 
$\alpha \leqslant \beta$ if $\beta -\alpha$ is a 
sum of positive roots. Then it is easy to see that $\Phi \in
\calf_{\emptyset}$ 
if and only if for all $\alpha \in \Phi, \beta \in \Delta^+$, such
that $\alpha \leqslant \beta$, then 
$\beta \in \Phi$. 

Let $\Phi\in\calf_{\emptyset}$. Set
$$
\Phi_{min}=\{\beta\in\Phi;\beta-\alpha\not\in\Phi, \mbox{ for all } \alpha\in \Delta^+\}.
$$
Then, $\Phi_{min}$ is an antichain of $\Delta^+$ with respect to the above partial order. Conversely, if we consider an antichain $\Gamma$, then, the set of roots which are bigger than any one of the elements of $\Gamma$ is an element of $\calf_{\emptyset}$.

As in \cite{CP}, we display the positive 
roots $\Delta^+$ in the Ferrers diagram $T_l$ of $(l,l-1,\dots,1)$ as follows : we assign to each box in the $i$-th
row and the $j$-th column, labelled $(i,j)$ in $T_l$, 
a positive root $t_{i,j}=\alpha_i+\cdots +\alpha_{l-j+1}$, $1\leqslant i,
j\leqslant l$. 

For example, for $l=5$, we have :
$$
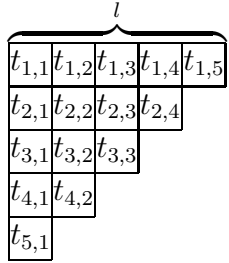
\overbrace{
\vbox{\offinterlineskip
\halign{#&#&#&#&#&#\cr
\htrait&\htrait&\htrait&\htrait&\htrait\cr
\sboite{$t_{1,1}$}&\sboite{$t_{1,2}$}&\sboite{$t_{1,3}$}&\sboite{$t_{1,4}$}&%
\sboite{$t_{1,5}$}&\vtrait\cr
\htrait&\htrait&\htrait&\htrait&\htrait\cr
\sboite{$t_{2,1}$}&\sboite{$t_{2,2}$}&\sboite{$t_{2,3}$}%
&\sboite{$t_{2,4}$}&\vtrait\cr
\htrait&\htrait&\htrait&\htrait\cr
\sboite{$t_{3,1}$}&\sboite{$t_{3,2}$}&\sboite{$t_{3,3}$}&\vtrait\cr
\htrait&\htrait&\htrait\cr
\sboite{$t_{4,1}$}&\sboite{$t_{4,2}$}&\vtrait\cr
\htrait&\htrait\cr
\sboite{$t_{5,1}$}&\vtrait\cr
\htrait\cr
}
}
}^{l}
$$

Observe that given two positive roots $\alpha,\beta$, $\alpha$ is bigger than or equal to $\beta$ if the box corresponding to $\alpha$ is in the quadrant north-west of the box corresponding to $\beta$. 
It follows easily that the map which sends an element $\Phi\in\calf_{\emptyset}$ to the subdiagram of $T_l$ consisting of the boxes corresponding to the roots of $\Phi$ defines a bijection between $\calf_{\emptyset}$ and the set of northwest flushed 
subdiagrams of $T_l$. Hence, by Section \ref{partition_Dyck_path}, we obtain a bijection $\sigma$ from $\calf_{\emptyset}$ to the set of $l$-partitions.

By Proposition \ref{map_D}, $D\circ\sigma$ is a bijection from $\calf_{\emptyset}$ to the set of Dyck paths of length $2l+2$.

For $\Phi\in\calf_{\emptyset}$, set 
$$
I_{\Phi}=\{\alpha\in\Pi;\Phi\in\calf_{\{\alpha\}}\}.
$$
It is the maximal element of $\{I\subset\Pi; \Phi\in\calf_I\}$.  We shall see how to link the number of $udu$ of the Dyck path 
$(D\circ\sigma)(\Phi)$ and the cardinality of $I_{\Phi}$.

Set $\alpha_{i,j}=\alpha_i +\dots+\alpha_j$, for all $1\leqslant i\leqslant j \leqslant l$. We have easily the following lemma :
\begin{lemme}\label{lemme_Phi_min}
Let $I\subset\Pi$. An element $\Phi\in\calf_{\emptyset}$ is an element of $\calf_I$ if and only if for all $\alpha_{i,j}\in\Phi_{min}$, we have $\alpha_i,\alpha_j\not\in I$.
\end{lemme}
It follows from Lemma \ref{lemme_Phi_min} that :
$$
I_{\Phi}=\Pi\setminus\{\alpha_i\in\Pi;\mbox{there exists } \alpha_{i,j} \mbox{ or } \alpha_{l,i}\in \Phi_{min}\}.
$$

The problem is not to count the same root twice. For example, in $A_7$, for $\Phi_{min}=\{\alpha_{1,3},\alpha_{2,5},\alpha_{5,7}\}$, we have $\Pi\setminus I_{\Phi}=\{\alpha_1,\alpha_2,\alpha_3,\alpha_5,\alpha_7\}$ but we find $\alpha_5$ in the beginning or in the end of the support of two roots in $\Phi_{min}$. So if we set :
$$
L=\{\alpha_{i,j}\in\Phi_{min};\mbox{ there exists a root of shape } \alpha_{p,i}\in\Phi_{min}\},
$$
$$
U=\Phi_{min}\setminus L,
$$
we obtain that 
\begin{equation}\label{cardinal_I_max}
\sharp I_{\Phi}=l-2\sharp U -\sharp L.
\end{equation}

Let $\lambda=\sigma(\Phi)$, $F$ its Ferrers diagram and $D_{\lambda}=D(\lambda)$ be the Dyck path which corresponds to $\lambda$ via the AKOP-bijection. Let $\alpha_{i,j}\in\Phi_{min}$. Then the cell $(i,l+1-j)=(i,\lambda_i)$ of $\alpha_{i,j}$ in $F$ is a south-east corner of the diagram and two cases are possible : there exists a rectangle $R_{p}$ such that $(i,\lambda_i)\in R_{p}$ or $(i,\lambda_i)$ is not in any rectangle. If the last case occurs, then $(i,l+1-j)$ is above a rectangle $R_p$. For example, if $\lambda=(5,3,1,1,1,0,0)$, we have that $\alpha_{2,5},\alpha_{5,7}$ are in the first case and $\alpha_{1,3}$ is in the second case.
$$
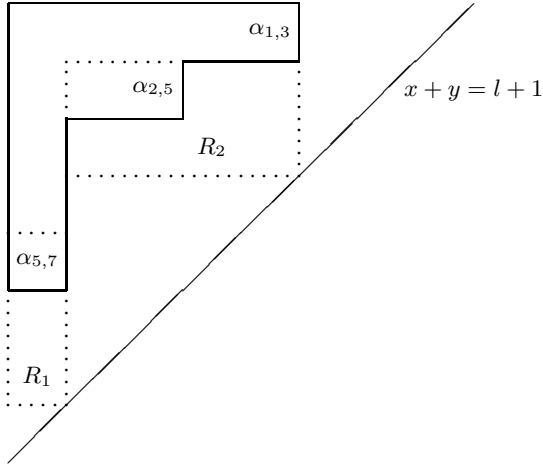
\begin{Tiles}{1.1}
\Dbloc{\Ttop\Tleft}\Dbloc{\Ttop\Tbottomleftdot\Tbottomdots}\Dbloc{\Ttop\Tbottomleftdot\Tbottomdots}\Dbloc{\Ttop\Tbottom}\Dbloc{\Ttop\Tbottom\Tright\Dtext{c}{\alpha_{1,3}}}\Dbloc{}\Dbloc{}\Dbloc{\Tantidiagonal}\Dskip
\Dbloc{\Tleft\Tbottomrightdot\Trightdots}\Dbloc{\Tbottom}\Dbloc{\Tright\Tbottom\Dtext{c}{\alpha_{2,5}}}\Dspace\Dspace\Dbloc{\Tleftdots\Tbottomleftdot}\Dbloc{\Tantidiagonal}\Dbloc{\Ddoubletext{c}{x+y=l+1}}\Dskip
\Dbloc{\Tleft\Tright}\Dbloc{\Tbottomleftdot\Tbottomdots}\Dbloc{\Tbottomleftdot\Tbottomdots}\Dbloc{\Tbottomleftdot\Tbottomdots\Dtext{c}{R_2}}
\Dbloc{\Tbottomleftdot\Tbottomdots\Tbottomrightdot\Trightdots}\Dbloc{\Tantidiagonal}\Dskip
\Dbloc{\Tleft\Tright\Tbottomdots\Tbottomrightdot\Tbottomleftdot}\Dspace\Dspace\Dspace\Dbloc{\Tantidiagonal}\Dskip
\Dbloc{\Tleft\Tright\Tbottom\Dtext{c}{\alpha_{5,7}}}\Dspace\Dspace\Dbloc{\Tantidiagonal}\Dskip
\Dbloc{\Tleftdots\Tbottomleftdot\Trightdots\Tbottomrightdot}\Dspace\Dbloc{\Tantidiagonal}\Dskip
\Dbloc{\Tleftdots\Tbottomleftdot\Trightdots\Tbottomrightdot\Tbottomdots\Dtext{c}{R_1}}\Dbloc{\Tantidiagonal}\Dskip
\Dbloc{\Tantidiagonal}\Dskip
\end{Tiles}
$$

If $\alpha_{i,j}$ is in the rectangle $R_{p}$, then the cell $(i,\lambda_i)=(i,l-j+1)$ which corresponds to $\alpha_{i,j}$ in $F$ satisfies :
\begin{equation}\label{equation_i}
l-i_{p+1}+2<i<l-i_p+2,
\end{equation}
\begin{equation}\label{equation_lambda}
i_{p-1}<\lambda_i\leqslant i_p,
\end{equation}
and so we have :
\begin{equation}\label{equation_j}
l-i_{p}+1\leqslant j< l-i_{p-1}+1.
\end{equation}

If $\alpha_{i,j}$ is above the rectangle $R_{p}$, then the cell $(i,l-j+1)$ which corresponds to $\alpha_{i,j}$ in $F$ satisfies :
\begin{equation}\label{equation_upper}
(i,l-j+1)=(l-i_{p+1}+2,i_p).
\end{equation}

Define the map $r$ from $\Phi_{min}$ to $\{1,\dots,k\}$ which to $\alpha_{i,j}$ associate the integer $r(\alpha_{i,j})=p$ such that $\alpha_{i,j}$ is in or immediately above the rectangle $R_{p}$.

Let $\alpha_{i,j}\in\Phi_{min}$ and $p=r(\alpha_{i,j})$. Since the cell $(i,l-j+1)$ which contains $\alpha_{i,j}$ in $T_l$ is a south-east corner, there is a horizontal line under this cell. If 
$c=(i,l-j+1)$ is in the rectangle $R_p$, then it is at the row $q=i-(l-i_{p+1}+2)$ of $R_p$ and the line under $c$ correspond to the part $l^{a_{p,q}}$ in $M_p$. Furthermore $(p,q)\in\cala_p$.

If $c$ is immediately above the rectangle $R_p$, then the line under $c$ corresponds to $l^{a_{p,0}}$ in $M_p$ and $(p,0)\in\cala_p$. Since in this case, by \eqref{equation_upper} we have $(i,l-j+1)=(l-i_{p+1}+2,i_p)$, we obtain that $i-(l-i_{p+1}+2)=0$. We can define in any case the map $s$ from $\Phi_{min}$ to $\nset$ by : 
\begin{equation}\label{equation_q}
s(\alpha_{i,j})=i-(l-i_{r(\alpha_{i,j})+1}+2).
\end{equation}
Furthermore, in both cases, the line under the cell which contains $\alpha_{i,j}$ is the part $l^{a_{r(\alpha_{i,j}),s(\alpha_{i,j})}}$ in $M_{r(\alpha_{i,j})}$ and $(r(\alpha_{i,j}),s(\alpha_{i,j}))\in\cala_{r(\alpha_{i,j})}$.

Conversely, let $(p,q)\in\cala_p$. Then, there is a horizontal line under the row $i=q-l-i_{p+1}+2$ of $F$ which is under a south-east corner of $F$. This south-east corner is a cell $(i,\lambda_i)$ which corresponds to a root $\alpha_{i,j}$, where $l-j+1=\lambda_i$. So we have a bijection :
$$
\begin{array}{ccrcl}
\Psi &:&\Phi_{min}&\rightarrow& \cala \\
&&\alpha_{i,j}&\mapsto & (r(\alpha_{i,j}),s(\alpha_{i,j}))
\end{array}
$$

\begin{lemme}
We have $\Psi(U)=\calu$ and $\Psi(L)=\call$. 
\end{lemme}

\begin{proof}
Since $L=\Phi_{min}\setminus U$ and $\call=\cala\setminus\calu$, it suffices to prove that $\Psi(L)=\call$. 

Let $\alpha_{i,j}\in L$. Set $p=r(\alpha_{i,j})$, $q=s(\alpha_{i,j})$ and let $c=(i,\lambda_i)$ be the cell which corresponds to $\alpha_{i,j}$ in $F$. 

First assume that $i=j$. Then, we have  $c=(i,l-i+1)$. If $c\in R_{p}$, then by \eqref{equation_i} and \eqref{equation_j}, we have :
$$
i=l-i_p+1,
$$
so by \eqref{equation_q}, we have that $q=i_{p+1}-i_p-1$ so by \eqref{equation_a_i}, $a_{p,q}\in\call_{p}$.

If $c$ is above $R_{p}$, then by \eqref{equation_upper}, we have $c=(i,l-i+1)=(l-i_{p+1}+2,i_p)$, so $q=0$ and $i_{p+1}-i_p=1$, hence by \eqref{equation_a_i} we also have $a_{p,q}\in\call_{p}$.

Now assume that $i\not=j$ and there exists a root of shape $\alpha_{m,i}\in\Phi_{min}$. Set $t=r(\alpha_{m,i})$. Let $(m,\lambda_m)=(m,l-i+1)$ be the cell which corresponds to $\alpha_{m,i}$ in $\lambda$. If $c\in R_{p}$, then by \eqref{equation_i}, we have 
$$
i_p\leqslant \lambda_m\leqslant i_{p+1}-2.
$$
So either $(m,\lambda_m)\in R_{p+1}$ or $(m,\lambda_m)=(l-i_{p+1}+2,i_p)$.

If $(m,\lambda_m)\in R_{p+1}$, then between the columns $i_{p+1}$ and $\lambda_m=l-i+1$, we have $i_{p+1}-(l-i+1)$ columns, so there exists $n$ such that $\sum_{u=0}^n a_{p+1,u}=i_{p+1}-(l-i+1)$. Furthermore, by \eqref{equation_q}, we have $q=i-(l-i_{p+1}+2)$, hence $a_{p,q}\in\call_{p}$.

If $(m,\lambda_m)=(l-i_{p+1}+2,i_p)$, then $i=l-i_p+1$ and by \eqref{equation_q}, we have that :
$$
q=(l-i_p+1)-(l-i_{p+1}+2)=i_{p+1}-i_p-1.
$$
Hence, by \eqref{equation_a_i}, we have $a_{p,q}\in\call_{p}$.

Conversely, let $a_{p,q}\in \call_p$, then there exists $0\leqslant t\leqslant h_{p+1}$ such that $q+1=\sum_{f=0}^t a_{p+1,f}$. There also exists $\alpha_{i,j}\in\Phi_{min}$ such that $r(\alpha_{i,j})=p$ and $s(\alpha_{i,j})=q$. By \eqref{equation_q}, we have that :
$$
q=i-(l-i_{p+1}+2).
$$
Observe that for all $0\leqslant j\leqslant h_{p+1}$, there exists a south-east corner $(n_j,\lambda_{n_j})$ in or above the rectangle $R_{p+1}$ such that :
$$
\lambda_{n_j}=i_{p+1}-\sum_{f=0}^j a_{p+1,f}.
$$
So there exists a south-east corner $(n_j,\lambda_{n_j})$ such that :
$$
\lambda_{n_j}=i_{p+1}-(q+1)=l-i+1.
$$
The element of $\Phi_{min}$ which corresponds to the cell $(n_j,\lambda_{n_j})$ is $\alpha_{n_j,i}$, so we have $\alpha_{i,j}\in L$.
\end{proof}

It follows by Proposition \ref{proposition_partition} and equation \eqref{cardinal_I_max} that we have the following theorem :
\begin{theoreme}
There is a bijection between the elements $\Phi\in\calf_{\emptyset}$ such that $\sharp I_{\Phi}=r$ and the Dyck paths of length $2l+2$ having $r$  ``$udu$''.
\end{theoreme}

Since the number of Dyck paths having a fixed number of $udu$ is calculated in \cite{Su}, we have the following corollary :
\begin{corollaire}\label{I_max_A}
The number of elements of $\Phi\in\calf_{\emptyset}$ such that $\sharp I_{\Phi}=r$ is
$$
\left(\begin{array}{c}
l\\
r
\end{array}
\right)
\sum_{k=0}^{[l-r]/2} \left(\begin{array}{c}
l-r\\
2k
\end{array}
\right)
\calc_{k}
$$
where $\calc_{k}$ denotes the $k$-th Catalan number.
\end{corollaire}

\begin{exemple}
Let $N_r^l$ be the number of elements $\Phi\in\calf_{\emptyset}$ such that $\sharp I_{\Phi}=r$. We have by Corollary \ref{I_max_A} :
$$
\begin{array}{|c||c|c|c|c|c|}
\hline
r &N_r^{1}&N_r^{2}&N_r^{3} &N_r^{4}&N_r^{5}\\
\hline
0 &1 &2 &4 &9 &21\\
\hline
1  &1 &2 &6 &16 &45 \\
\hline
2&  &1 &3 &12 &40\\
\hline
3&& &1& 4& 20\\
\hline
4&&&& 1& 5\\
\hline
5& & & &&1 \\
\hline
\end{array}
$$
\end{exemple}


\section{Duality}\label{section_dualite}
We shall construct a duality between the elements of $\calf_{\emptyset}$ such that $\sharp \Phi_{min}=p$ and those such that $\sharp \Phi_{min}=l-p$.

\begin{prop}\label{prop_phi_min1}
Let $\Phi\in\calf_{\emptyset}$. Let $N$ be the number of peaks in $(D\circ\sigma)(\Phi)$, then we have :
$$
\sharp \Phi_{min}=l-(N-1).
$$
\end{prop}

\begin{proof}
Let $\lambda=\sigma(\Phi)$ be the corresponding $l$-partition. Recall that the construction of $D(\lambda)$ is iterative. At each step, when we add $a_{p,q}$ peaks to a 
highest peak, for $(p,q)\in\cala_p$, we also ``destroy'' the initial highest peak. So, we add only $a_{p,q}-1$ peaks. At the end of the construction we have :
$$
l-\lambda_1+1+\sum_{p=1}^k \sum_{(p,q)\in\cala_p}(a_{p,q}-1)
$$
peaks. Since $\sum_{p=1}^k \sum_{(p,q)\in\cala_p}a_{p,q}=\sum_{(p,q)\in\cala} a_{p,q}=\lambda_1$ and $\cala$ is in bijection with $\Phi_{min}$ by section \ref{ideaux_Dyck_path}, we obtain the result.
\end{proof}

\begin{prop}\label{prop_phi_min2}
Let $\Phi\in\calf_{\emptyset}$ and $p$ be the number of peaks in $(P\circ\sigma)(\Phi)$, then we have :
$$
\sharp \Phi_{min}=p-1.
$$
\end{prop}

\begin{proof}
The result is clear by the construction of $(P\circ\sigma)(\Phi)$ defined in Section \ref{partition_Dyck_path}.
\end{proof}

\begin{theoreme}
The map $\sigma^{-1}\circ P^{-1}\circ D\circ\sigma$ induces a bijection from $\calf_{\emptyset}$ to $\calf_{\emptyset}$ which sends 
$\Phi\in\calf_{\emptyset}$ such that $\sharp \Phi_{min}=p$ to $\Psi\in\calf_{\emptyset}$ such that $\sharp \Psi_{min}=l-p$.
\end{theoreme}

For example, in $sl_4(\cset)$, the element $\Phi=\{\theta\}\in\calf_{\emptyset}$ corresponds to the partition $\lambda=(1,0,0)$, and the Dyck path $D_{\lambda}$ is :
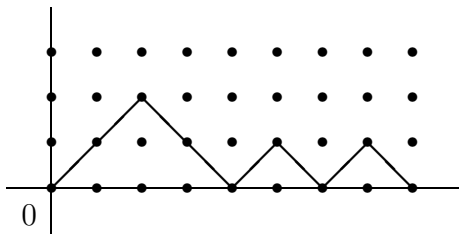
\begin{center}
$$
\Gitter(9,4)(0,0)
\Koordinatenachsen(9,4)(0,0)
\Pfad(0,0),33443434\endPfad
\Label\lu{0}(0,0)
\hskip4.5cm
$$
\end{center}
Then, $P^{-1}(D_{\lambda})=(3,2,0)$ which is the partition which corresponds to $\Psi$ such that $\Psi_{min}=\{\alpha_1, \alpha_2\}$.

\begin{remarque}
It was  proved in \cite{P} that when $\glie$ is a simple Lie algebra of type $A$ and $C$, the number of elements 
$\Phi\in\calf_{\emptyset}$ such that 
$\sharp\Phi_{min}=p$ is the same as the number of elements $\Phi\in\calf_{\emptyset}$ such that $\sharp\Phi_{min}=l-p$. 
But the duality of \cite{P} is not the same as the one defined above. For example, in $sl_4(\cset)$, if we consider $\Phi=\{\theta\}$ like above, the dual ideal defined by \cite{P} is $\Psi$ where $\Psi_{min}=\{\alpha_1+\alpha_2, \alpha_3\}$.

\end{remarque}


\end{document}